\documentclass[reqno,a4paper]{cas-sc}
\usepackage{amssymb}
\usepackage{enumitem}

\usepackage{chngcntr}
\counterwithin*{equation}{section}
\counterwithin*{equation}{subsection}
\usepackage{stmaryrd}
\usepackage[utf8]{inputenc}
\usepackage{csquotes}
\usepackage[english]{babel}
\usepackage{amsthm}
\usepackage{mathtools}
\usepackage[all]{xy}
\usepackage{tikz-cd}

\usepackage{color}

\newtheorem{theorem}{Theorem}[section]

\newtheorem{lemma}[theorem]{Lemma}
\newtheorem{definition}[theorem]{Definition}
\newtheorem{definitionlemma}[theorem]{Definition-Lemma}
\newtheorem{proposition}[theorem]{Proposition}
\newtheorem{remark}[theorem]{Remark}
\newtheorem{example}[theorem]{Example}

\let\originalmiddle=\middle
\renewcommand{\middle}[1]{\mathrel{}\originalmiddle#1\mathrel{}}

\newcommand{\st}[1]{\star_{\scalebox{0.4}{\!\!$#1$}}}

\newcommand{\ee}[2]{{\sf \underline{e}}_{\scalebox{0.4}{$#1,#2$}}}
\newcommand{\id}[1]{\operatorname{id}_{\scalebox{0.4}{$#1$}}}
\newcommand{\e}[1]{{ \sf e}_{\scalebox{0.4}{$#1$}}}

\newcommand{\str}[2]{\overset{\bigstar}{\scriptscriptstyle M_{#1}}_{{#2}=1}^{n}}
\DeclareMathOperator{\Aut}{Aut}
\DeclareMathOperator{\Hom}{Hom}
\DeclareMathOperator{\End}{End}

\DeclareMathOperator{\Nil}{Nil}

\begin{document}

\title{Automorphism groups of direct products of multiplicative monoids of certain rings}
\author[1]{Joseph Atalaye}[orcid=0009-0005-7329-2889]
\ead{26828146@sun.ac.za}

\author[1,2]{Liam Baker}[orcid=0000-0003-2728-3963]
\ead{liambaker@sun.ac.za}
\ead[URL]{https://math.sun.ac.za/liambaker}
\cormark[1]

\author[3,4]{Sophie Marques}[orcid=0000-0001-8700-3375]
\ead{smarques@math.uminho.pt}
\ead[URL]{https://sites.google.com/site/sophiemarques64/}

\affiliation[1]{
    organisation={Department of Mathematical Sciences,
    University of Stellenbosch},
    city={Stellenbosch},
    country={South Africa}
}
\affiliation[2]{
    organisation={NITheCS (National Institute for Theoretical and Computational Sciences)},
    country={South Africa}
}
\affiliation[3]{
    organisation={Department of Mathematics,
    University of Minho},
    city={Minho},
    country={Portugal}
}
\affiliation[4]{
    organisation={CMAT (Centre of Mathematics)},
    country={Portugal}
}

\begin{keywords}
    Multiplicative automorphism \sep monoid \sep group theory \sep algebraic structures \sep unit group \sep ring theory
\end{keywords}

\begin{abstract}
In this paper, we establish a rigidity result for automorphisms of multiplicative direct products of $D$-rings which are total ring of fraction that have pairwise distinct cardinalities.
Under these assumptions, every automorphism acts independently on each factor, so that no interaction between distinct components occurs; in particular, the automorphism group decomposes canonically as the direct product of the automorphism groups of the factors.
As a consequence, the automorphism group of the multiplicative monoid of integers modulo $n$ is entirely determined by its $p$-power components.
\end{abstract}

\maketitle

\setcounter{tocdepth}{3}
\tableofcontents
\raggedbottom 

\section*{Introduction}

The determination of automorphism groups is a central problem in algebra, as it encodes the intrinsic symmetries of algebraic structures and reveals to what extent their global behaviour is governed by local data.
This problem becomes particularly subtle for direct products, where automorphisms may act independently on each factor, permute components, or involve more intricate interactions between them.

A fundamental question is therefore to determine which of these behaviours actually occur in a given setting.
This question is especially relevant for multiplicative monoids arising in arithmetic, such as $(\mathbb{Z}/n\mathbb{Z},\cdot)$, where understanding the global symmetry structure is closely related to the way homomorphisms factor through $p$-power components (see \cite{zbMATH01650459, zbMATH06968189, zbMATH03549232}).

The study of automorphism groups of direct products has been extensively developed in group theory.
In a series of works, Bidwell, together with Curran and McCaughan, analyzed automorphisms of direct products of finite groups \cite{bidwell2006automorphisms, bidwell2008automorphisms}.
Independently, Francis \cite{zbMATH03780218} investigated related questions in the context of surface groups.

These developments naturally suggest analogous questions in the broader setting of monoids.
In this direction, Karimi and Doostie \cite{ahmadidelir2012automorphisms} and Brescia \cite{zbMATH08024120} studied automorphisms of direct products of monoids and semigroups, obtaining structural descriptions in specific cases.

In our previous work \cite{JoeSophieLiam}, we carried out a detailed analysis of the multiplicative automorphism group of the $p$-power modular monoid $(\mathbb{Z}/p^e\mathbb{Z},\cdot)$, obtaining explicit structural decompositions.
While this provides a complete understanding of the local case, it leaves open the global question of how these symmetries behave under direct products.

\medskip

The main purpose of the present paper is to show that, in a natural and broad class of multiplicative monoids, no such interactions occur.
More precisely, we establish a rigidity theorem for automorphisms of direct products of multiplicative monoids associated with $D$-rings which are total rings of fractions.
We prove that for a family of such rings with \emph{pairwise distinct cardinalities}, every automorphism acts independently on each factor.
In particular, we obtain a canonical decomposition
\[
\Aut(R_1 \times \dotsm \times R_n,\cdot)
\cong
\Aut(R_1,\cdot)\times \dotsm \times \Aut(R_n,\cdot).
\]

Thus, despite the potential for complex behaviour, the global automorphism group is entirely determined by its local components: all off-diagonal contributions are forced to vanish.
The assumption on distinct cardinalities plays a decisive role in preventing interactions between different factors, thereby enforcing a diagonal structure.

\medskip

A primary motivation for this result is the arithmetic case of the multiplicative monoid $(\mathbb{Z}/n\mathbb{Z},\cdot)$, which admits the canonical decomposition
\[
(\mathbb{Z}/n\mathbb{Z},\cdot)\cong \prod_{p^e \parallel n} (\mathbb{Z}/p^e\mathbb{Z},\cdot).
\]
Our theorem shows that this decomposition is fully reflected at the level of automorphisms, yielding
\[
\Aut(\mathbb{Z}/n\mathbb{Z},\cdot)
\cong
\prod_{p^e \parallel n}\Aut(\mathbb{Z}/p^e\mathbb{Z},\cdot).
\]


More generally, the result applies to a large and natural class of rings.
In the finite case, $D$-rings which are total rings of fractions coincide with finite local rings, and thus include classical examples such as $\mathbb{Z}/p^e\mathbb{Z}$, truncated polynomial rings $k[x]/(P^n)$ for $k$ a finite field and irreducible $P \in k[x]$, and mixed nilpotent extensions such as $\mathbb{Z}/p^e\mathbb{Z}[x]/(P^n)$.
In the infinite case, the same condition characterizes local rings whose maximal ideal coincides with the nilradical, thereby extending the scope of the theorem beyond the Artinian setting.

This shows that the arithmetic situation is not isolated, but rather a manifestation of a general structural principle governing a wide class of multiplicative monoids, provided that the components can be distinguished by their cardinalities.

\medskip

Our approach is based on representing endomorphisms of direct products of monoids by matrices whose entries are homomorphisms between the factors.
The core of the argument is to show that all off-diagonal components of an automorphism necessarily vanish, thereby forcing a diagonal form and yielding the desired decomposition.

The paper is organized as follows.

\noindent
In Section~1, we recall the necessary background on zero divisors, $D$-rings, and total rings of fractions, and establish their basic structural properties.
In particular, we characterize $D$-rings which are total rings of fractions.

\noindent
In Section~2, we develop a matrix description of endomorphisms of direct products of monoids, in which the entries are homomorphisms between the components.
In Subsection~\ref{Subsec21}, we characterize trivial homomorphisms between rings and $D$-rings (Proposition~\ref{Prop21}).
Subsection~\ref{Subsec232} is devoted to structural properties of these matrices, including criteria for invertibility and irreversibility (Proposition~\ref{Prop23} and Lemma~\ref{Lem24}).

\noindent
In Section~3, we establish the main rigidity result for direct products of multiplicative monoids associated with $D$-rings which are total rings of fractions.
Subsection~\ref{Sec51} contains the key structural analysis: we show that all off-diagonal components of an automorphism are trivial (Lemma~\ref{Lem33}), while the diagonal components are necessarily non-trivial (Lemma~\ref{Lem34}).
In Subsection~\ref{Subsec32}, we prove the main theorem (Theorem~\ref{Theo35}) and conclude with Example~\ref{Exp522}, which illustrates the breadth of the result and its applicability beyond the arithmetic setting.

\section{Preliminaries}\label{Sec1}

Throughout this paper, $\mathbb{N}$ denotes the set of natural numbers including $0$, $p$ a prime number, and $e$ a positive integer.
For simplicity, we write $[a]$ for $[a]_{p^e}$ and adopt the convention $[a]^0 = 1$ for all $a \in \mathbb{Z}$.
For integers $a$ and $b$, we denote by $\llbracket a,b \rrbracket = \{a,a+1,\dots,b\}$.

\medskip

The goal of this section is to collect the structural properties that will be used throughout the paper.
We begin with introducing the class of rings that naturally governs our setting.

\subsection{$D$-rings which are total rings of fractions}

The main results of this paper are formulated for rings satisfying a strong structural constraint: every element behaves either as a unit or as a nilpotent element.
This dichotomy is captured by the notion of a $D$-ring which is a total ring of fractions.
In this section, we describe the structure of such rings, showing that in the finite case they coincide with finite local rings, while in the infinite case they correspond to local rings whose maximal ideal consists entirely of nilpotent elements.

\medskip

\begin{definition}
Let $R$ be a ring.
\begin{enumerate}
    \item The set of zero divisors of $R$ is
    \[
    \mathcal{Z}_R \coloneq  \{\, z \in R \mid \exists\, t \in R\setminus\{0\} \text{ such that } zt = 0\,\}.
    \]
    \item The ring $R$ is called a \emph{$D$-ring} if every zero divisor of $R$ is nilpotent.
    \item The ring $R$ is called a \emph{total ring of fractions} if every non-zero divisor of $R$ is a unit.
\end{enumerate}
\end{definition}

\medskip

These two conditions combine to impose a very rigid structure.

\begin{proposition}\label{Prop:unit_or_nilpotent}
Let \(R\) be a $D$-ring which is a total ring of fractions.
Then every element of \(R\) is either a unit or nilpotent.
\end{proposition}

\begin{proof}
Let \(a\in R\).
If \(a\) is a zero divisor, then it is nilpotent since \(R\) is a $D$-ring.
Otherwise, \(a\) is not a zero divisor, hence a unit since \(R\) is a total ring of fractions.
\end{proof}


\begin{remark}
Proposition~\ref{Prop:unit_or_nilpotent} shows that in such rings there is a sharp dichotomy between invertible elements and nilpotent elements.
This dichotomy is the key structural feature underlying the rigidity phenomena studied in this paper.
\end{remark}

\medskip

We now show that, in the finite commutative setting, this condition is equivalent to locality.

\begin{proposition}\label{Prop:finite_total_ring_of_fractions}
Let \(R\) be a finite commutative ring.
Then every element of \(R\) is either a unit or a zero divisor.
In particular, every finite commutative ring is a total ring of fractions.
\end{proposition}

\begin{proof}
Let \(a\in R\).
If \(a\) is not a zero divisor, then the multiplication map $x\mapsto ax$ is injective.
Since \(R\) is finite, it is surjective.
Hence there exists \(x\in R\) such that \(ax=1\), so \(a\) is a unit.
\end{proof}

\begin{theorem}\label{thm:finite_D_rings_local}
Let \(R\) be a finite commutative ring.
Then \(R\) is a $D$-ring if and only if \(R\) is local.
\end{theorem}

\begin{proof}
Assume that \(R\) is a $D$-ring.
Since \(R\) is finite, it is Artinian.
By the structure theorem for Artinian rings, we have
\[
R \cong R_1 \times \dotsm \times R_t,
\]
with each \(R_i\) local.
If \(t\geq 2\), the element \(e=(1,0,\dotsc,0)\) is a non-nilpotent zero divisor, a contradiction.
Hence \(R\) is local.

Conversely, if \(R\) is local with maximal ideal \(\mathfrak m\), then every non-unit lies in \(\mathfrak m\).
By Proposition~\ref{Prop:finite_total_ring_of_fractions}, every element is either a unit or a zero divisor, so the zero divisors are precisely \(\mathfrak m\).
Since \(R\) is Artinian, \(\mathfrak m\) is nilpotent, hence every zero divisor is nilpotent.
\end{proof}

Theorem~\ref{thm:finite_D_rings_local} shows that, in the finite case, $D$-rings which are total rings of fractions are exactly finite local rings.
In particular, their structure is well understood.

\medskip

In the finite case, the previous results reduce the study of $D$-rings which are total rings of fractions to the well-understood class of finite local rings.

\medskip

Let \((R,\mathfrak{m})\) be a finite local ring.
Then \(\mathfrak{m}\) is nilpotent and the residue ring \(R/\mathfrak{m}\) is a finite field.
In particular, \(R\) may be viewed as a finite field together with a nilpotent thickening, whose structure is controlled by the successive quotients
\[
\mathfrak{m}^i/\mathfrak{m}^{i+1}.
\]


\begin{example}\label{ex:finite_local_models}
Typical examples of finite local rings include:
\begin{itemize}
    \item the rings \(\mathbb{Z}/p^n\mathbb{Z}\), where \(p\) is prime,
    \item the rings \(k[\varepsilon]/(\varepsilon^n)\), where \(k\) is a finite field.
\end{itemize}
In the second case, one recovers the field \(k\) when \(n=1\), while for \(n \geq 2\) one obtains non-trivial nilpotent extensions.
These examples illustrate how finite local rings arise by adjoining nilpotent elements to a field.
\end{example}
While the previous examples arise from a single nilpotent parameter, the general situation may involve several independent nilpotent directions.
\medskip

\begin{example}[Non-principal nilpotent structure]\label{ex:non_principal_local}
Let \(k\) be a finite field.
The ring
\[
R = k[x,y]/(x^2,y^2,xy)
\]
is a finite local ring with maximal ideal \(\mathfrak m=(\bar x,\bar y)\) and residue field \(k\).

In this case, the maximal ideal is not generated by a single element.
More precisely, one has
\[
\mathfrak m^2 = 0,
\]
but \(\mathfrak m\) requires two generators.
This example shows that finite local rings need not be principal, and that their nilpotent structure may involve several independent directions.
\end{example}
The following example illustrates a more intricate nilpotent structure, combining both arithmetic and polynomial nilpotent directions.
\begin{example}[Mixed nilpotent structure]
Let \(p\) be a prime and \(n,m \geq 1\).
The ring
\[
R = (\mathbb{Z}/p^n\mathbb{Z})[x]/(x^m)
\]
is a finite local ring with maximal ideal
\[
\mathfrak m = (p,x)
\]
and residue field \(\mathbb{F}_p\).

This example combines two independent nilpotent directions: the arithmetic nilpotent element \(p\) and the geometric nilpotent element \(x\).
In particular, the maximal ideal is not principal, illustrating a more complex nilpotent structure.
\end{example}

We now turn to the infinite case.
While the defining conditions remain the same, the structure becomes more flexible.

\begin{theorem}\label{thm:infinite_D_total_structure}
Let \(R\) be a commutative ring.
The following are equivalent:
\begin{enumerate}
    \item \(R\) is a $D$-ring and a total ring of fractions;
    \item \(R\) is a local ring whose maximal ideal coincides with its nilradical.
\end{enumerate}
\end{theorem}

\begin{proof}
Assume that \(R\) is a $D$-ring which is a total ring of fractions.
By Proposition~\ref{Prop:unit_or_nilpotent}, every element is either a unit or nilpotent.
Thus the set of non-units is precisely \(\Nil(R)\), which is an ideal.
Hence \(R\) is local with maximal ideal \(\Nil(R)\).

Conversely, if \(R\) is local with maximal ideal \(\mathfrak m=\Nil(R)\), then every zero divisor lies in \(\mathfrak m\) and is therefore nilpotent.
Moreover, every non-zero divisor is not in \(\mathfrak m\), hence is a unit.
\end{proof}

\begin{example}
Let \(k\) be a field and set
\[
R = k[x_1,x_2,x_3,\dots]/(x_1^2,x_2^3,x_3^4,\dots).
\]
Then \(R\) is a local ring whose maximal ideal consists entirely of nilpotent elements, but is not nilpotent.
Hence \(R\) is a $D$-ring which is a total ring of fractions, but is not Artinian.
\end{example}

\smallskip

\subsection{Monoid-theoretic preliminaries}

We now recall basic facts about monoids and their homomorphisms.
The proofs are straightforward and follow directly from the definitions.

\begin{lemma}\label{Lem227}\cite[Lemma 1.3]{JoeSophieLiam}
Let $(M,\star_M,e_M)$ and $(N,\star_N,e_N)$ be monoids and $\phi \in \Hom((M,\star_M),(N,\star_N))$.
Then:
\begin{enumerate}
\item $\phi(M^\times) \subseteq N^\times$.
If $\phi$ is an automorphism, then $\phi(M^\times)=N^\times$.
\item If $M$ and $N$ have absorbing elements $0_M$ and $0_N$, and $0_N \in \operatorname{Im}(\phi)$, then $\phi(0_M)=0_N$.
\end{enumerate}
\end{lemma}

\begin{definition} \label{def:default_maps}
    Let $(M,\star_M,e_M)$ and $(N,\star_N,e_N)$ be monoids.
    Then we define the constant `zero' monoid morphism $\ee{M}{N} : M \to N, \ x \mapsto e_N$ and the identity morphism $\id{M} : M \to M, \ x \mapsto x$.
\end{definition}

\begin{definition}
Let $(M,\star_M,e_M)$ be a monoid.
Its center is
\[
Z(M) \coloneq  \{ m \in M \mid m \star_M n = n \star_M m \text{ for all } n \in M \}.
\]
\end{definition}

\begin{lemma}\label{Lem2110}
Let $(M,\star_M,e_M)$ be a monoid.
Then:
\begin{enumerate}
\item $Z(M)$ is a submonoid of $M$;
\item $Z(M)^\times = M^\times \cap Z(M)$.
\end{enumerate}
\end{lemma}

\begin{proof}
(1) is clear.
For (2), let $x\in M^\times\cap Z(M)$, and let $x^{-1}$ be its inverse.
For every $m\in M$, from $xm=mx$ we obtain $m x^{-1}=x^{-1}m$, so $x^{-1}\in Z(M)$, and so $x\in Z(M)^\times$.
The reverse inclusion is immediate, since $Z(M)^\times \subseteq M^\times$ and $Z(M)^\times \subseteq Z(M)$.
\end{proof}
\noindent

\subsection*{Matrix description of endomorphisms}

The study of automorphisms of direct products has been extensively developed in group theory, notably in the work of J.
N.
S.
Bidwell.
In particular, in \cite{bidwell2008automorphisms}, Bidwell described the endomorphism monoid of a direct product of $n$ finite groups in terms of matrices whose entries are homomorphisms between the factors.

This matrix perspective extends naturally to the setting of monoids.
We now recall the corresponding construction for finite direct products of monoids, following \cite[Section~2]{zbMATH08024120}.
\medskip

In order to express matrix multiplication explicitly, we extend the product of a monoid to maps in a pointwise manner.
Let $(M,\star_M,e_M)$ be a monoid, and let $f_1,\dots,f_n : X \to M$ be maps.
We define
\[ \overset{\bigstar}{\scriptscriptstyle M}_{{k}=1}^{n} f_k : X \longrightarrow M
\qquad
\text{by}
\qquad
\left(\overset{\bigstar}{\scriptscriptstyle M}_{{k}=1}^{n} f_k\right)(x)
=
f_1(x)\star_M \dotsb \star_M f_n(x),
\qquad x \in X.
\]
This operation is well-defined by associativity of $\star_M$.

\medskip

We also recall the following notation for commutation.
Let $A,B \subseteq M$.
We set
\[
Z(A,B)
\coloneq 
\{(a,b) \in A \times B \mid a \star_M b = b \star_M a\}.
\]
Thus $Z(A,B)=A\times B$ if and only if every element of $A$ commutes with every element of $B$.
\medskip

\begin{definitionlemma}\cite[Section 2]{zbMATH08024120}\label{DefLem:n-fold-M}
Let $(M \coloneq  M_1 \times \dotsm \times M_n, \star)$ be a finite direct product of monoids.
Define
\[
\mathcal{M}(M)
=
\left\{
(\alpha_{ij})_{1 \le i,j \le n}
~\middle|
\begin{array}{l}
\alpha_{ij} \in \Hom(M_j, M_i),\\[0.3em]
Z(\operatorname{Im}(\alpha_{ij_1}), \operatorname{Im}(\alpha_{ij_2}))
=
\operatorname{Im}(\alpha_{ij_1}) \times \operatorname{Im}(\alpha_{ij_2}),\\[0.3em]
\forall\, i,\; j_1 \neq j_2
\end{array}
\right\}.
\]
We equip $\mathcal{M}(M)$ with the binary operation $\odot$ defined by
\[
(\alpha \odot \alpha')_{ij}
=
\str{i}{k} \alpha_{ik} \circ \alpha'_{kj}.
\]
This operation is well-defined thanks to the centrality condition on the images.

Then:
\begin{enumerate}
\item $(\mathcal{M}(M), \odot)$ is a monoid with identity element
\[
I_n = (\varepsilon_{ij})_{1\le i,j \le n},
\quad
\varepsilon_{ij} =
\begin{cases}
\id{M_i}, & i=j,\\
\ee{M_j}{M_i}, & i\neq j;
\end{cases}
\]

\item for all $\alpha \in \mathcal{M}(M)$, all $m=(m_1,\dots,m_n)$, $m'=(m'_1,\dots,m'_n)\in M$, and all $i \in \llbracket 1,n \rrbracket$, we have
\begin{equation}\label{commutativity}
\str{i}{k} \alpha_{ik}(m_k \star_{M_k} m'_k)
=
\left(\str{i}{k} \alpha_{ik}(m_k)\right)
\star_{M_i}
\left(\str{i}{k} \alpha_{ik}(m'_k)\right).
\end{equation}
\end{enumerate}
\end{definitionlemma}

For the proof, one may consult \cite[Proposition~1.3]{zbMATH03780218} and adapt the argument to the monoid setting.

\medskip

The next theorem shows that this matrix construction captures all endomorphisms of $M$.

\begin{theorem}\cite[Proposition 2.1]{zbMATH08024120}\label{Theo252}
Let $(M \coloneq  M_1 \times \dotsm \times M_n, \star)$ be a finite direct product of monoids.
The map
\[
\begin{array}{cccc}
\Psi: & \End(M,\star) & \longrightarrow & \mathcal{M}(M) \\
& \theta & \longmapsto & (\theta_{ij})
\end{array}
\]
is a monoid isomorphism, where for all $i,j \in \llbracket 1,n \rrbracket$,
\[
\theta_{ij} \coloneq  \pi_i \circ \theta \circ \iota_j.
\]

Here, $\iota_j: M_j \to M$ and $\pi_i: M \to M_i$ denote the canonical injection and projection, respectively.

The inverse map is given by
\[
\Psi^{-1}((\theta_{ij}))(m)
=
\bigl(\theta_1(m),\dots,\theta_n(m)\bigr),
\quad
\text{where}
\quad
\theta_i(m) \coloneq \str{i}{k} \theta_{ik}(m_k)
\quad\text{for}\ m=(m_1,\dots,m_n)\in M.
\]
\end{theorem}

For the proof, see \cite[Lemma~2.1]{bidwell2006automorphisms} and \cite[Proposition~1.4]{zbMATH03780218} in the group case, and adapt the argument to the monoid setting.

We conclude this section with a simple identity that will be used repeatedly in the sequel.

\begin{lemma}\label{Lem33}
Let \((M\coloneq M_1 \times \dotsm \times M_n, \star)\) be a finite direct product of monoids.
Then
\[
\str{}{\ell} \iota_{\ell}\pi_{\ell}=\id{M},
\]
where \(\pi_{\ell}\) and \(\iota_{\ell}\) are as defined in Theorem~\ref{Theo252} for all \(\ell\in \llbracket 1, n\rrbracket.\)
\end{lemma}

\begin{proof}
Let \(m=(m_1,\dots,m_n)\in M\).
For each \(\ell\in \llbracket 1,n\rrbracket\), we have
\[
\iota_\ell\pi_\ell(m)
=
(\e{M_1},\dots,\e{M_{\ell-1}},m_\ell,\e{M_{\ell+1}},\dots,\e{M_n}).
\]
Taking the product over all \(\ell\), we obtain the result.
\end{proof}

\section{Endomorphisms of direct products of monoids as matrices}\label{Sec23}

\subsection{Characterization of trivial morphisms between monoids}\label{Subsec21}

We conclude this section with a structural result describing multiplicative homomorphisms between rings, and in particular their interaction with zero divisors.

\begin{proposition}\label{Prop21}
Let $(R,+,\cdot)$ and $(S,+,\cdot)$ be rings, and let
\[
\psi \in \Hom\big((R,\cdot),(S,\cdot)\big)
\]
be a multiplicative monoid homomorphism.
Then the following hold:

\begin{enumerate}
\item The following are equivalent:
\begin{enumerate}
    \item[(a)] $\psi = \ee{R}{S}$;
    \item[(b)] $\psi(0_R)=1_S$;
    \item[(c)] $\operatorname{Im}(\psi)$ contains no zero divisors.
\end{enumerate}

\item The following are equivalent:
\begin{enumerate}
    \item[(a)] $\psi \neq \ee{R}{S}$;
    \item[(b)] $\psi(0_R)\in \mathcal{Z}_S$.
\end{enumerate}

\item Assume in addition that $S$ is a $D$-ring.
Then the following are equivalent:
\begin{enumerate}
    \item[(a)] $\psi \neq \ee{R}{S}$;
    \item[(b)] $\operatorname{Im}(\psi)$ contains a zero divisor;
    \item[(c)] $\psi(0_R)=0_S$;
    \item[(d)] $\psi(\mathcal{Z}_R)\subseteq \mathcal{Z}_S$.
\end{enumerate}
\end{enumerate}
\end{proposition}

\begin{proof}
\begin{enumerate}
\item We show that $(a)\implies(b)\implies(c)\implies(a)$.

\begin{itemize}
\item $(a)\implies(b)$ is immediate.

\item $(b)\implies(c)$: if $\psi(0_R)=1_S$, then for all $x\in R$,
\[
\psi(x)=\psi(x)\psi(0_R)=\psi(x\cdot 0_R)=\psi(0_R)=1_S,
\]
so $\operatorname{Im}(\psi)=\{1_S\}$ contains no zero divisors.

\item $(c)\implies(a)$: for all $x\in R$, $\psi(0_R)\psi(x)=\psi(0_R)$, hence $\psi(0_R)(\psi(x)-1_S)=0_S$.
Since $\operatorname{Im}(\psi)$ contains no zero divisors, we must have $\psi(x)=1_S$ for all $x$, hence $\psi=\ee{R}{S}$.
\end{itemize}

\item We show $(a)\iff(b)$.

\begin{itemize}
\item $(a)\implies(b)$: if $\psi\neq \ee{R}{S}$, there exists $x$ such that $\psi(x)\neq 1_S$.
Then
\[
\psi(0_R)(\psi(x)-1_S)=0_S
\]
with $\psi(x)-1_S\neq 0_S$, so $\psi(0_R)\in \mathcal{Z}_S$.

\item $(b)\implies(a)$: if $\psi(0_R)$ is a zero divisor, then $\psi(0_R)\neq 1_S$, hence $\psi\neq \ee{R}{S}$.
\end{itemize}

\item Assume $S$ is a $D$-ring.
We show $(a)\implies(b)\implies(c)\implies(d)\implies(a)$.

\begin{itemize}
\item $(a)\implies(b)$ follows from part (2).

\item $(b)\implies(c)$: let $t=\psi(z)$ be a zero divisor in the image.
Since $S$ is a $D$-ring, $t^k=0_S$ for some $k$.
Hence
\(
\psi(z^k)=t^k=0_S,
\)
so $0_S\in \operatorname{Im}(\psi)$.
Then
\[
\psi(z^k)\psi(0_R)=\psi(0_R),
\]
which gives $0_S\cdot \psi(0_R)=\psi(0_R)$, hence $\psi(0_R)=0_S$.

\item $(c)\implies(d)$: if $z\in \mathcal{Z}_R$, then $z^k=0_R$ for some $k$, so
\(
\psi(z)^k=\psi(0_R)=0_S,
\)
so $\psi(z)\in \mathcal{Z}_S$.

\item $(d)\implies(a)$: since $0_R\in \mathcal{Z}_R$, we get $\psi(0_R)\in \mathcal{Z}_S$, hence $\psi\neq \ee{R}{S}$ by part (2).
\qedhere
\end{itemize}
\end{enumerate}
\end{proof}

\subsection{Characterisation of automorphisms of direct products of monoids}\label{Subsec232}
We now turn to the description of the automorphism group of the product of monoids.
By Theorem~\ref{Theo252}, every endomorphism of $M$ can be represented by a matrix whose entries are homomorphisms between the components.
In this framework, automorphisms correspond precisely to those endomorphisms whose associated matrices are invertible, that is, admit an inverse in $\mathcal{M}(M)$.
Our aim is therefore to make these invertibility conditions explicit.

We start by relating an automorphism \(\theta \in \Aut(M, \star)\) to its inverse \(\theta^{-1}\) through their matrix representations in \(\mathcal{M}(M)\).
The following result is a direct generalisation of \cite[Lemma 2.4]{zbMATH08024120}, and we provide a proof for completeness.

\begin{proposition}\label{Prop23}
Let \((M\coloneq M_1 \times \dotsm \times M_n, \star)\) be a direct product of monoids.
Then \(\theta \in \Aut(M, \star)\) if and only if there exists \((\theta'_{ij})_{1\leq i,j\leq n} \in \mathcal{M}(M)\) such that for all \(m=(m_1,\cdots,m_n)\in M\), and all \(i, j \in \llbracket 1,n \rrbracket\) with \(i\neq j\),
\begin{align*}
    (1)~\str{j}{\ell}\theta_{j\ell}\theta'_{\ell j} &= \id{M_j}, & (1')~\str{i}{\ell}\theta_{i\ell}\theta'_{\ell j} &= \ee{M_j}{M_i}, & (2)~\str{j}{\ell}\theta'_{j\ell}\theta_{\ell j} &= \id{M_j}, & (2')~\str{i}{\ell}\theta'_{i\ell}\theta_{\ell j} &= \ee{M_j}{M_i}.
\end{align*}
where the family $(\theta_{i,j})_{1 \leq i,j \leq n}$ is defined with respect to $\theta$ as in Theorem~\ref{Theo252}.
\end{proposition}

\begin{proof}
Assume first that \(\theta \in \Aut(M,\star)\), and let
\((\theta_{ij})\) and \(((\theta^{-1})_{ij})\) be the matrices associated with \(\theta\) and \(\theta^{-1}\) as in Theorem~\ref{Theo252}.
Then
\[
\theta\theta^{-1}
=
\left(\str{i}{\ell}\theta_{i\ell}(\theta^{-1})_{\ell j}\right)_{i,j}.
\]
Using Lemma~\ref{Lem33}, we compute
\[
\str{i}{\ell}\theta_{i\ell}(\theta^{-1})_{\ell j}
=\str{i}{\ell}\pi_i\theta\iota_\ell\pi_\ell\theta^{-1}\iota_j
=\pi_i\theta\left(\str{}{\ell}\iota_\ell\pi_\ell\right)\theta^{-1}\iota_j
=\pi_i\theta\theta^{-1}\iota_j
=\pi_i\iota_j
=
\begin{cases}
\id{M_j}, & i=j,\\
\ee{M_j}{M_i}, & i\neq j.
\end{cases}
\]
A similar computation applied to \(\theta^{-1}\theta\) yields relations \((2)\) and \((2')\), so the required conditions hold with \(\theta'_{ij}\coloneq (\theta^{-1})_{ij}\).

Conversely, suppose there exists $(\theta'_{ij}) \in \mathcal{M}(M)$ satisfying relations $(1)$--$(2')$.
Then, by definition of \(\odot\), these identities exactly express that
\[
(\theta_{ij}) \odot (\theta'_{ij}) = I_n
\quad\text{and}\quad
(\theta'_{ij}) \odot (\theta_{ij}) = I_n.
\]
Thus $(\theta_{ij})$ is invertible in $\mathcal{M}(M)$, and since $\Psi$ is an isomorphism (Theorem~\ref{Theo252}), it follows that \(\theta \in \Aut(M,\star)\).
\end{proof}

Consequently, Lemma~\ref{Prop23} yields the following structural relations, which describe how units in the monoids $M_i$ arise from the matrix components of an automorphism.

\begin{lemma}\label{Lem24}
Let \((M\coloneq M_1 \times \dotsm \times M_n, \star)\) be a finite direct product of monoids, and let \(\theta \in \Aut(M, \star)\).
Let $(\theta_{ij})_{1 \leq i,j \leq n}$ (resp.
$((\theta^{-1})_{ij})_{1 \leq i,j \leq n}$) be the family associated with $\theta$ (resp.
$\theta^{-1}$) as in Theorem~\ref{Theo252}.

Then, for all \(m=(m_1,\dotsc,m_n)\in M\) and all \(i,j,k \in \llbracket 1,n \rrbracket\) with \(i \neq j\), we have:

\begin{enumerate}
\item \(\theta_{ik}(\theta^{-1})_{k j}(m_j)\in M_i^{\times}\), and
\[
\theta_{ik}(\theta^{-1})_{k j}(m_j)
=
\left[\overset{\bigstar}{\scriptscriptstyle M_i}_{\ell=1,\ell\neq k}^n
\theta_{i\ell}(\theta^{-1})_{\ell j}(m_j)\right]^{(-1)};
\]

\item \((\theta^{-1})_{ik}\theta_{k j}(m_j)\in M_i^{\times}\), and
\[
(\theta^{-1})_{ik}\theta_{k j}(m_j)
=
\left[\overset{\bigstar}{\scriptscriptstyle M_i}_{\ell=1,\ell\neq k}^n
(\theta^{-1})_{i\ell}\theta_{\ell j}(m_j)\right]^{(-1)}.
\]
\end{enumerate}

We distinguish between the notation $m^{(-1)}$, which denotes the inverse of an element $m$ in a monoid, and $\theta^{-1}$, which denotes the inverse of an automorphism.
\end{lemma}

\begin{proof}
Fix \(i,j,k \in \llbracket 1,n \rrbracket\) with \(i \neq j\), and let \(m_j \in M_j\).

\begin{enumerate}
\item From equation \((1')\) of Lemma~\ref{Prop23}, we have
\[
\str{i}{\ell}\theta_{i\ell}(\theta^{-1})_{\ell j}(m_j)=\e{M_i}.
\]
Separating the term corresponding to \(k\), we obtain
\[
\e{M_i}
=
\theta_{ik}(\theta^{-1})_{k j}(m_j)
\st{M_i}
\left[\overset{\bigstar}{\scriptscriptstyle M_i}_{\ell=1,\ell\neq k}^n
\theta_{i\ell}(\theta^{-1})_{\ell j}(m_j)\right].
\]
Thus \(\theta_{ik}(\theta^{-1})_{k j}(m_j)\) admits a right inverse, hence lies in \(M_i^\times\), and the stated identity follows.

\item The proof is entirely similar, using equation \((2')\) of Lemma~\ref{Prop23} instead of \((1')\). \qedhere
\end{enumerate}
\end{proof}

In the preceding results, we have shown that the unit elements of the monoids \(M_i\), for \(i \in \llbracket 1,n \rrbracket\), can be described in terms of compositions of homomorphisms between the components.
We now make this interaction more precise by relating the diagonal and off-diagonal components of an automorphism.

\begin{lemma}\label{Lem25}
Let \( (M \coloneq  M_1 \times \dotsm \times M_n, \star) \) be a finite direct product of monoids, and let \( \theta \in \Aut(M, \star) \).
Let $(\theta_{ij})_{1 \leq i,j \leq n}$ (resp.
$((\theta^{-1})_{ij})_{1 \leq i,j \leq n}$) be the family associated with $\theta$ (resp.
$\theta^{-1}$) as in Theorem~\ref{Theo252}, and fix \(i,j \in \llbracket 1,n \rrbracket\) with $i \neq j$.

Then the following hold:
\begin{enumerate}
    \item \label{Lem25:item1} If \( \theta_{ii} \in \Aut(M_i, \st{M_i}) \), then
    \[
        \theta_{ij} \in \Hom(M_j, Z(M_i))
        \quad \text{and} \quad
        (\theta^{-1})_{ij} \in \Hom(M_j, M_i^{\times}).
    \]
    
    \item If, moreover, both \( \theta_{ii} \) and \( (\theta^{-1})_{ii} \) are automorphisms of \( M_i \), then
    \[
        \theta_{ij} \in \Hom(M_j, Z(M_i)^{\times})
        \quad \text{and} \quad
        (\theta^{-1})_{ij} \in \Hom(M_j, Z(M_i)^{\times}).
    \]
\end{enumerate}
\end{lemma}

\begin{proof}
Assume first that \( \theta_{ii} \in \Aut(M_i, \st{M_i}) \).
By Definition–Lemma~\ref{DefLem:n-fold-M}, we have
\[
Z(\operatorname{Im}(\theta_{ii}), \operatorname{Im}(\theta_{ij}))
=
\operatorname{Im}(\theta_{ii}) \times \operatorname{Im}(\theta_{ij}).
\]
Since \(\operatorname{Im}(\theta_{ii}) = M_i\), it follows that $\operatorname{Im}(\theta_{ij}) \subseteq Z(M_i)$, and hence \(\theta_{ij} \in \Hom(M_j, Z(M_i))\).

Next, by Lemma~\ref{Lem24}, for all \(m_j \in M_j\),
\[
\theta_{ii}\bigl((\theta^{-1})_{ij}(m_j)\bigr) \in M_i^\times.
\]
Since \(\theta_{ii}\) is an automorphism, it preserves invertibility, and therefore $(\theta^{-1})_{ij}(m_j) \in M_i^\times$.
Thus \((\theta^{-1})_{ij} \in \Hom(M_j, M_i^\times)\), which proves item \ref{Lem25:item1}.

Assume now that both \( \theta_{ii} \) and \( (\theta^{-1})_{ii} \) are automorphisms of \(M_i\).
Applying the same argument to \(\theta^{-1}\), we obtain
\[
(\theta^{-1})_{ij} \in \Hom(M_j, Z(M_i))
\quad \text{and} \quad
\theta_{ij} \in \Hom(M_j, M_i^\times).
\]
Combining these inclusions with item \ref{Lem25:item1}, it follows by Lemma~\ref{Lem2110} that both maps take values in
\[
Z(M_i) \cap M_i^\times = Z(M_i)^\times. \qedhere
\]
\end{proof}

\section{Rigidity conditions for automorphisms of direct products}\label{Sec3}

In this section, we establish a rigidity phenomenon for automorphisms of direct products of multiplicative monoids arising from a broad class of rings.
Our goal is to understand how the components of an automorphism interact across different factors, and to show that, under suitable structural assumptions, no non-trivial interaction between distinct components can occur.

Although the motivating examples arise from $p$-power modular monoids, the results are formulated and proved in a more general setting, which applies to a wider class of rings, including finite $D$-rings.
In particular, the arguments developed here do not rely on specific arithmetic properties, but rather on general structural features such as the behaviour of zero divisors and the invertibility of non-zero divisors.

\subsection{Key lemmas}\label{Sec51}

We begin with two structural lemmas that form the core of the rigidity argument.
They show that, in the presence of invertibility constraints, the off-diagonal components of an automorphism are forced to be trivial.
\begin{lemma}\label{Lem:idempotent}
Let \(R_i\) be rings for all \(i\in \llbracket 1,n\rrbracket\), and let \(\theta \in \Aut(R_1\times \dotsm \times R_n, \star)\) be a multiplicative monoid automorphism.
Then, for all \(i, j, k \in \llbracket 1,n\rrbracket\) with \(i\neq j\), we have
\[
\theta_{ik}(\theta^{-1})_{k j} = \ee{R_j}{R_i}
\qquad
\text{and}
\qquad
(\theta^{-1})_{ik}\theta_{k j} = \ee{R_j}{R_i},
\]
where the family $(\theta_{ij})_{1 \leq i,j \leq n}$ is defined with respect to $\theta$ as in Theorem~\ref{Theo252}.
\end{lemma}

\begin{proof}
Fix \(i,j \in \llbracket 1,n\rrbracket\) with \(i\neq j\).

From Proposition~\ref{Prop23}, equation \((1')\), we have 
\[
\str{i}{k}\theta_{ik}(\theta^{-1})_{kj}(m_j) = \e{R_i}
\qquad\text{for all } m_j \in R_j.
\]
By Lemma~\ref{Lem24}, each factor
$\theta_{ik}(\theta^{-1})_{kj}(m_j)$
is invertible in \(R_i\).
Therefore, for each $k \in \llbracket{1,n}\rrbracket$, the image of
$\theta_{ik}(\theta^{-1})_{kj}$
contains no zero divisors.
By Proposition~\ref{Prop21}, it follows that
$\theta_{ik}(\theta^{-1})_{kj} = \ee{R_j}{R_i}$.

The second identity is proved in the same way, using equation \((2')\) of Proposition~\ref{Prop23}.
\end{proof}

Building on the previous lemma, we now investigate the properties of the indices for which the zero element is preserved.
\begin{lemma}\label{Lem:Dring_zero}
Let \(R_i\) be \(D\)-rings which are total rings of fractions for all \(i \in \llbracket 1,n\rrbracket\).

Let \(\theta \in \Aut(R_1 \times \dotsm \times R_n, \star)\), and let \((\theta_{ij})_{1 \leq i,j \leq n}\) be the associated family defined as in Theorem~\ref{Theo252}.

Then, for each \(r \in \llbracket 1,n\rrbracket\), there exist indices \(c(r),d(r) \in \llbracket 1,n\rrbracket\) such that
\begin{align*}
    (\theta^{-1})_{c(r)r}(0_{R_r}) &= 0_{R_{c(r)}},
    & \theta_{r c(r)}(0_{R_{c(r)}}) &= 0_{R_r}, &
    \theta_{d(r)r}(0_{R_r}) &= 0_{R_{d(r)}},
    & (\theta^{-1})_{r d(r)}(0_{R_{d(r)}}) &= 0_{R_r}.
\end{align*}
\end{lemma}

\begin{proof}
Fix \(r \in \llbracket 1,n\rrbracket\).

By equation~\((1)\) of Proposition~\ref{Prop23} evaluated at $0_{R_r}$ we have
\[
\theta_{r1}(\theta^{-1})_{1r}(0_{R_r})
\st{R_r}
\theta_{r2}(\theta^{-1})_{2r}(0_{R_r})
\st{R_r}
\cdots
\st{R_r}
\theta_{rn}(\theta^{-1})_{nr}(0_{R_r})
=0_{R_r}.
\]
Hence at least one factor is not a unit in \(R_r\).
Therefore, there exists
\(c(r)\in \llbracket 1,n\rrbracket\) such that
\[
\theta_{r c(r)}(\theta^{-1})_{c(r)r}(0_{R_r})\notin R_r^\times.
\]

Now \((\theta^{-1})_{c(r)r}(0_{R_r})\in R_{c(r)}\), and since \(R_{c(r)}\) is a total ring of fractions, it is either a unit or a zero divisor.
Suppose that \((\theta^{-1})_{c(r)r}(0_{R_r})\) is a unit.
Then, by Lemma~\ref{Lem227},
\[
\theta_{r c(r)}\big((\theta^{-1})_{c(r)r}(0_{R_r})\big)\in R_r^\times,
\]
a contradiction, so it is a zero divisor.
By Proposition~\ref{Prop21} equation (2), it follows that
\[
(\theta^{-1})_{c(r)r}(0_{R_r})=0_{R_{c(r)}}.
\]
Moreover, Proposition~\ref{Prop21} equation (2) also gives
\[
\theta_{r c(r)}(0_{R_{c(r)}})=0_{R_r}.
\]

The proof of the existence of \(d(r)\in\llbracket 1,n\rrbracket\) such that
\[
\theta_{d(r)r}(0_{R_r})=0_{R_{d(r)}}
\quad\text{and}\quad
(\theta^{-1})_{r d(r)}(0_{R_{d(r)}})=0_{R_r}
\]
is analogous, using equation~\((2)\) of Proposition~\ref{Prop23} in place of equation~\((1)\), and exchanging \(\theta\) and \(\theta^{-1}\).
\end{proof}

The previous lemma identifies, for each component, indices along which the zero element is preserved.
The following result shows that these indices in fact determine the behaviour of all the component maps, reducing the structure to a single nontrivial entry in each row and column.
\begin{lemma}\label{Lem331}
Let \(R_i\) be \(D\)-rings which are total rings of fractions for all \(i\in \llbracket 1,n\rrbracket\).
Let \(\theta \in \Aut(R_1\times \dotsm \times R_n, \star)\), and let \((\theta_{ij})_{1 \leq i,j \leq n}\) be the associated family defined as in Theorem~\ref{Theo252}.

Then there exist maps
\[
c,d:\llbracket 1,n\rrbracket \longrightarrow \llbracket 1,n\rrbracket
\]
such that:
\begin{enumerate}
    \item For each \(r\in \llbracket 1,n\rrbracket\), we have
    \begin{align*}
        \theta_{r c(r)}(0_{R_{c(r)}}) &= 0_{R_r}, & (\theta^{-1})_{c(r)r}(0_{R_r}) &= 0_{R_{c(r)}}, & \theta_{r c(r)}(\theta^{-1})_{c(r)r} &= \id{R_r}, \\
        (\theta^{-1})_{r d(r)}(0_{R_{d(r)}}) &= 0_{R_r}, & \theta_{d(r)r}(0_{R_r}) &= 0_{R_{d(r)}}, & (\theta^{-1})_{r d(r)}\theta_{d(r)r} &= \id{R_r}.
    \end{align*}
    and for all $i \in \llbracket{1,n}\rrbracket$ with \(i\neq r\),
    \[
    \theta_{i c(r)}=\ee{R_{c(r)}}{R_i}, \quad
    \theta_{c(r)i}=\ee{R_i}{R_{c(r)}}, \quad
    (\theta^{-1})_{i d(r)}=\ee{R_{d(r)}}{R_i}, \quad\text{and}\quad
    (\theta^{-1})_{d(r)i}=\ee{R_i}{R_{d(r)}}.
    \]
    \item the maps \(c\) and \(d\) are injective.
\end{enumerate}
\end{lemma}

\begin{proof}
\begin{enumerate}
\item Fix \(i,r\in \llbracket 1,n\rrbracket\) with $i \neq r$.

By Lemma~\ref{Lem:Dring_zero}, there exists \(c(r)\in \llbracket 1,n\rrbracket\) such that
\[
\theta_{r c(r)}(0_{R_{c(r)}})=0_{R_r}.
\]
By Lemma~\ref{Lem:idempotent}, we have
$\theta_{i c(r)}(\theta^{-1})_{c(r)r}=\ee{R_r}{R_i}$.
Evaluating at \(0_{R_r}\), and using Lemma~\ref{Lem:Dring_zero}, we obtain
$\theta_{i c(r)}(0_{R_{c(r)}})=\e{R_i}$.
Hence, by Proposition~\ref{Prop21} equation (1), we have
$\theta_{i c(r)}=\ee{R_{c(r)}}{R_i}$.

Moreover, Lemma~\ref{Lem:idempotent} gives
$\theta_{r c(r)}(\theta^{-1})_{c(r)i} = \ee{R_i}{R_r}$.
Evaluating at \(0_{R_i}\), we obtain
\[
\theta_{r c(r)}\big((\theta^{-1})_{c(r)i}(0_{R_i})\big)=\e{R_r}.
\]
Since \(R_{c(r)}\) is a total ring of fractions, the element \((\theta^{-1})_{c(r)i}(0_{R_i})\) is either a unit or a zero divisor.
If it were a zero divisor, then by Proposition~\ref{Prop21} equations (2)--(3), it would be equal to \(0_{R_{c(r)}}\), and therefore
$\theta_{r c(r)}(0_{R_{c(r)}}) = \e{R_r}$,
contradicting the choice of \(c(r)\).
Thus \((\theta^{-1})_{c(r)i}(0_{R_i})\) is a unit, and Proposition~\ref{Prop21} equation (2) yields
$(\theta^{-1})_{c(r)i}=\ee{R_i}{R_{c(r)}}$.

Finally, equation~\((1)\) of Proposition~\ref{Prop23} yields
\[
\theta_{r c(r)}(\theta^{-1})_{c(r)r}=\id{R_r},
\]
since \(\theta_{r\ell}=\ee{R_\ell}{R_r}\) for all \(\ell\neq c(r)\).

Applying the same argument to \(\theta^{-1}\) in place of \(\theta\), we obtain an index \(d(r)\in \llbracket 1,n\rrbracket\) as in the statement.

\item We first show that \(c\) is injective.
Suppose that \(c(r)=c(r')\) for some
\(r\neq r'\).
Then part~(1), applied to \(r\), gives
\[
\theta_{r' c(r)}=\ee{R_{c(r)}}{R_{r'}}
\qquad
\text{and so}
\qquad
\theta_{r' c(r)}(0_{R_{c(r)}})=\e{R_{r'}}.
\]
On the other hand, part~(1), applied to \(r'\), gives
\[
\theta_{r' c(r')}(0_{R_{c(r')}})=0_{R_{r'}}.
\]
Since \(c(r)=c(r')\), this is a contradiction.
Thus \(c\) is injective.

The injectivity of \(d\) is proved in exactly the same way, replacing \(\theta\) with \(\theta^{-1}\).
\qedhere
\end{enumerate}
\end{proof}

The previous lemma shows that the automorphism is governed by a system of distinguished indices, with exactly one nontrivial component in each row and column.
This strongly suggests an underlying permutation-type structure.
As a first consequence, we prove that no diagonal component can be trivial.

\begin{lemma}\label{Lem34}
Let \(R_i\) be \(D\)-rings which are total rings of fractions for all \(i \in \llbracket 1,n\rrbracket\), and assume that the \(R_i\) have pairwise distinct cardinalities.

Let \(\theta \in \Aut(R_1\times \dotsm \times R_n, \star)\), and let \((\theta_{ij})_{1 \leq i,j \leq n}\) be the associated family defined as in Theorem~\ref{Theo252}.

Then, for all \(i\in \llbracket 1,n\rrbracket\), the maps \(\theta_{ii}\) and \((\theta^{-1})_{ii}\) are nontrivial monoid homomorphisms, that is,
\[
\theta_{ii}\neq \ee{R_i}{R_i}
\qquad \text{and} \qquad
(\theta^{-1})_{ii}\neq \ee{R_i}{R_i}.
\]
\end{lemma}

\begin{proof}
By Lemma~\ref{Lem331}, for each \(r\in \llbracket 1,n\rrbracket\), there exists \(c(r)\in \llbracket 1,n\rrbracket\) such that $\theta_{r c(r)}(\theta^{-1})_{c(r)r} = \id{R_r}$.

Suppose for contradiction that there exists \(r\in \llbracket 1,n\rrbracket\) such that
$\theta_{rr} = \ee{R_r}{R_r}$.
Then necessarily \(c(r)\neq r\).
Since
\[
\theta_{r c(r)}(\theta^{-1})_{c(r)r} = \id{R_r}
\]
by Lemma \ref{Lem331}, the map $(\theta^{-1})_{c(r)r}:R_r\to R_{c(r)}$ is injective and so $|R_r|\leq |R_{c(r)}|$.
Since $c(r) \neq r$, by Lemma~\ref{Lem331} we have $\theta_{c(r)c(r)} = \ee{R_{c(r)}}{R_{c(r)}}$; additionally, since the cardinalities are pairwise distinct, we have
\[
|R_r|<|R_{c(r)}|.
\]

Iterating this argument, we obtain
\[
|R_r|<|R_{c(r)}|<|R_{c^2(r)}|<\dotsb<|R_{c^n(r)}|.
\]
Since the indices take values in a set of size \(n\), two of them must coincide, yielding a contradiction.


The statement for \((\theta^{-1})_{ii}\) follows by applying the same argument to \(\theta^{-1}\).
\end{proof}

\subsection{The main theorem}\label{Subsec32}

The automorphism group of a direct product of $p$-power multiplicative modular monoids decomposes as the product of the automorphism groups of its components.

\begin{theorem}\label{Theo35}
Let \(R_i\) be \(D\)-rings which are total rings of fractions, for all \(i \in \llbracket 1,n\rrbracket\), and assume that the \(R_i\) have pairwise distinct cardinalities.
Let $(M,\star)= \left(R_1\times \dotsm \times R_n, \star \right)$.
Then
\[
\Aut(M,\star) \cong \mathcal{A}(M,\star)
\]
via the isomorphism \(\Psi\) described in Theorem~\ref{Theo252}, where
\[
\mathcal{A}(M, \star) =
\left\{ (\theta_{ij})_{1 \leq i,j \leq n} ~\middle|~
\theta_{ij} =
\begin{cases}
\theta_{ii} \in \Aut(R_i, \star_{R_i}) & \text{if } i = j, \\
\ee{R_j}{R_i} & \text{if } i \ne j
\end{cases}
\right\}.
\]
In particular,
\[
\Aut(M, \star) \cong \Aut(R_1, \star_{R_1}) \times \dotsm \times \Aut(R_n, \star_{R_n}).
\]
\end{theorem}

\begin{proof}
Let \(\theta \in \Aut(M, \star)\).
By Theorem~\ref{Theo252}, the map \(\Psi\) is injective and $\Psi(\theta) = (\theta_{ij})_{1 \leq i,j \leq n}$, where each
\[
\theta_{ij} \in \Hom\bigl( (R_{j}, \star_{R_j}), (R_{i}, \star_{R_i}) \bigr).
\]

Fix \(i,j \in \llbracket 1,n\rrbracket\) with \(i \neq j\).
By Lemma~\ref{Lem:idempotent}, we have $\theta_{ij}(\theta^{-1})_{jj} = \ee{R_j}{R_i}$.
By Lemma~\ref{Lem34}, the map \((\theta^{-1})_{jj}\) is nontrivial.
Since \(R_j\) is a \(D\)-ring which is a total ring of fractions, Proposition~\ref{Prop21} yields
\[
(\theta^{-1})_{jj}(0_{R_j})=0_{R_j}.
\]
Hence
\[
\theta_{ij}(0_{R_j})
=
\theta_{ij}\bigl((\theta^{-1})_{jj}(0_{R_j})\bigr)
=
\ee{R_j}{R_i}(0_{R_j})
=
\e{R_i}.
\]
Therefore, by Proposition~\ref{Prop21}, we conclude that
\[
\theta_{ij}=\ee{R_j}{R_i}
\qquad\text{for all } i\neq j.
\]

We now determine the diagonal entries.
Fix \(i\in \llbracket 1,n\rrbracket\).
By equation~\((1)\) of Proposition~\ref{Prop23}, we have
\[
\str{i}{\ell}\theta_{i\ell}(\theta^{-1})_{\ell i}=\id{R_i}.
\]
For every \(\ell\neq i\), we have already shown that $\theta_{i\ell}=\ee{R_\ell}{R_i}$.
Hence the above identity reduces to
\[
\theta_{ii}(\theta^{-1})_{ii}=\id{R_i}.
\]
Similarly, equation~\((2)\) of Proposition~\ref{Prop23} gives
\[
\str{i}{\ell}(\theta^{-1})_{i\ell}\theta_{\ell i}=\id{R_i}.
\]
Again, for every \(\ell\neq i\), the maps \(\theta_{\ell i}\) are trivial, so this reduces to
\[
(\theta^{-1})_{ii}\theta_{ii}=\id{R_i}.
\]
Thus \(\theta_{ii}\) is invertible, with inverse \((\theta^{-1})_{ii}\), and therefore
\[
\theta_{ii}\in \Aut(R_i,\star_{R_i}).
\]

We have shown that $\Psi(\Aut(M,\star)) \subseteq \mathcal A(M,\star)$.
Conversely, let \((\varphi_{ij})_{1\leq i,j\leq n} \in \mathcal A(M,\star)\).
By definition, $\varphi_{ii}\in \Aut(R_i,\star_{R_i})$ and $\varphi_{ij}=\ee{R_j}{R_i}$ for $i\neq j$.
Define
\[
\varphi : M \longrightarrow M,
\qquad
(x_1,\dots,x_n)\longmapsto \bigl(\varphi_{11}(x_1),\dots,\varphi_{nn}(x_n)\bigr).
\]
Then \(\varphi\) is an automorphism of \((M,\star)\), and \(\Psi(\varphi)=(\varphi_{ij})\).
Hence $\mathcal A(M,\star)\subseteq \Psi(\Aut(M,\star))$.

Therefore,
\[
\Psi(\Aut(M,\star))=\mathcal A(M,\star), \qquad\text{and so}\qquad \Aut(M,\star)\cong \mathcal A(M,\star).
\]

The final statements follow immediately.
\end{proof}

\begin{example}\label{Exp522}
By \S 1.1, finite \(D\)-rings which are total rings of fractions are precisely finite local rings.
Typical examples therefore include:
\begin{itemize}
    \item the rings \(k[x]/(x^n)\), where \(k\) is a finite field and \(n \geq 1\),
    \item the rings \(\mathbb{Z}/p^{e}\mathbb{Z}\), where \(p\) is prime and \(e \geq 1\),
    \item more generally, mixed nilpotent examples such as
    \[
    \mathbb{Z}/p^{e}\mathbb{Z}[x]/(x^m), \qquad e,m \geq 1.
    \]
\end{itemize}

In the first case, one recovers the field \(k\) when \(n=1\), while for \(n \geq 2\) one obtains non-trivial nilpotent extensions.
In all these cases, the rings are finite local, hence finite \(D\)-rings which are total rings of fractions.

\medskip

Theorem~\ref{Theo35} therefore applies to any finite product of such rings whose cardinalities are pairwise distinct.
For instance,
\[
\Aut\!\bigl(\mathbb Z/2\mathbb Z \times \mathbb Z/4\mathbb Z \times \mathbb Z/8\mathbb Z,\cdot\bigr)
\cong
\Aut(\mathbb Z/2\mathbb Z,\cdot)\times
\Aut(\mathbb Z/4\mathbb Z,\cdot)\times
\Aut(\mathbb Z/8\mathbb Z,\cdot),
\]
showing that the theorem does not require the factors to involve distinct primes.

One may also combine different finite families.
For example,
\[
\Aut\!\bigl(\mathbb Z/4\mathbb Z \times \mathbb F_2[x]/(x^3)\times \mathbb Z/27\mathbb Z,\cdot\bigr)
\cong
\Aut(\mathbb Z/4\mathbb Z,\cdot)\times
\Aut(\mathbb F_2[x]/(x^3),\cdot)\times
\Aut(\mathbb Z/27\mathbb Z,\cdot).
\]

Also, for primes \(p_1,\dots,p_n\) and integers \(e_1,\dots,e_n \geq 1\) such that the cardinalities \(p_1^{e_1},\dots,p_n^{e_n}\) are pairwise distinct, one obtains
\[
\Aut\!\bigl( \mathbb{Z}/p_1^{e_1}\mathbb{Z} \times \dotsm \times \mathbb{Z}/p_n^{e_n}\mathbb{Z},\, \cdot \bigr) 
\cong 
\Aut(\mathbb{Z}/p_1^{e_1}\mathbb{Z}, \cdot) \times \dotsm \times \Aut(\mathbb{Z}/p_n^{e_n}\mathbb{Z}, \cdot).
\]

\medskip

The scope of Theorem~\ref{Theo35} is not restricted to the finite case.
By Theorem~\ref{thm:infinite_D_total_structure}, infinite \(D\)-rings which are total rings of fractions are precisely local rings whose maximal ideal coincides with the nilradical.
Thus Theorem~\ref{Theo35} also applies to products involving finite and infinite factors, provided that their cardinalities are pairwise distinct.
For example, if
\[
R_1=\mathbb{Z}/4\mathbb{Z}
\qquad\text{and}\qquad
R_2 = k[x_1,x_2,x_3,\dots]/(x_1^2,x_2^3,x_3^4,\dots),
\]
then \(R_1\) and \(R_2\) are \(D\)-rings which are total rings of fractions with distinct cardinalities, and one obtains
\[
\Aut(R_1\times R_2,\cdot)\cong \Aut(R_1,\cdot)\times \Aut(R_2,\cdot).
\]
\end{example}

\begin{example}\label{Exp523}
The hypothesis on pairwise distinct cardinalities in Theorem~\ref{Theo35} is essential.

Indeed, consider
\[
M = \mathbb Z/4\mathbb Z \times \mathbb Z/4\mathbb Z
\qquad
\text{and the automorphism}
\qquad
\tau : M \longrightarrow M,
\quad
(x,y) \longmapsto (y,x)
\]
of \((M,\cdot)\) which exchanges the two factors.

However, \(\tau\) is not of the form described in Theorem~\ref{Theo35}, since it is not diagonal: its associated family \((\tau_{ij})\) satisfies
\[
\tau_{12} = \id{\mathbb Z/4\mathbb Z},
\qquad
\tau_{21} = \id{\mathbb Z/4\mathbb Z},
\qquad
\tau_{11} = \ee{\mathbb Z/4\mathbb Z}{\mathbb Z/4\mathbb Z},
\qquad
\text{and}
\qquad
\tau_{22} = \ee{\mathbb Z/4\mathbb Z}{\mathbb Z/4\mathbb Z}.
\]

Thus,
\[
\Aut(\mathbb Z/4\mathbb Z \times \mathbb Z/4\mathbb Z,\cdot)
\not\cong
\Aut(\mathbb Z/4\mathbb Z,\cdot)\times \Aut(\mathbb Z/4\mathbb Z,\cdot),
\]
showing that the conclusion of Theorem~\ref{Theo35} fails when the factors have the same cardinality.
\end{example}

\bibliographystyle{plain}
\bibliography{refs}

@article{JoeSophieLiam,
    title={On the automorphism group of the monoid of the integers modulo a prime power},
    author={Atalaye, Joseph and Baker, Liam and Marques, Sophie},
    eprint={2408.06278},
    journal={arXiv},
    primaryClass={math.RA},
    year={2024},
    url={https://arxiv.org/abs/2408.06278}
}

@article {bidwell2006automorphisms,
    AUTHOR = {Bidwell, J. N. S. and Curran, M. J. and McCaughan, D. J.},
     TITLE = {Automorphisms of direct products of finite groups},
   JOURNAL = {Archiv der Mathematik},
    VOLUME = {86},
      YEAR = {2006},
    NUMBER = {6},
     PAGES = {481--489},
   MRCLASS = {20D45 (20D40)},
  MRNUMBER = {2241597},
MRREVIEWER = {Ma.\ Jes\'us\ Iranzo},
       DOI = {10.1007/s00013-005-1547-z}
}

@article{zbMATH03780218,
 author = {Johnson, Francis E. A.},
 title = {Automorphisms of direct products of groups and their geometric realisations},
 journal = {Mathematische Annalen},
 issn = {0025-5831},
 volume = {263},
 pages = {343--364},
 year = {1983},
 doi = {10.1007/BF01457136},
 zbMATH = {3780218},
 Zbl = {0495.57008}
}

@article {zbMATH08024120,
    AUTHOR = {Brescia, Mattia},
     TITLE = {A determinant for automorphisms of groups},
   JOURNAL = {Communications in Algebra},
    VOLUME = {53},
      YEAR = {2025},
    NUMBER = {6},
     PAGES = {2484--2509},
   MRCLASS = {20F28 (20E36 20H99)},
  MRNUMBER = {4885066},
MRREVIEWER = {Gary\ L.\ Peterson},
       DOI = {10.1080/00927872.2024.2446540}
}

@article {bidwell2008automorphisms,
    AUTHOR = {Bidwell, J. N. S.},
     TITLE = {Automorphisms of direct products of finite groups. {II}},
   JOURNAL = {Archiv der Mathematik},
    VOLUME = {91},
      YEAR = {2008},
    NUMBER = {2},
     PAGES = {111--121},
   MRCLASS = {20D45},
  MRNUMBER = {2430793},
MRREVIEWER = {Ma.\ Jes\'us\ Iranzo},
       DOI = {10.1007/s00013-008-2653-5}
}

@article {ahmadidelir2012automorphisms,
    AUTHOR = {Ahmadidelir, Karim and Doostie, Hossein},
     TITLE = {On the automorphisms of direct product of monogenic semigroups and monoids},
   JOURNAL = {Turkish Journal of Mathematics},
    VOLUME = {36},
      YEAR = {2012},
    NUMBER = {1},
     PAGES = {95--99},
      ISSN = {1300-0098,1303-6149},
   MRCLASS = {20M15 (20B25)},
  MRNUMBER = {2881640},
MRREVIEWER = {Ahmet\ Sinan\ \c Cevik},
       DOI = {10.3906/mat-1003-172}
}

@book{zbMATH01650459,
  title={Mathematical foundations of computational engineering: a handbook},
  author={Pahl, Peter J and Damrath, Rudolf},
  year={2012},
  publisher={Springer Science \& Business Media},
  DOI={10.1007/978-3-642-56893-0}
}

@book {zbMATH06968189,
    AUTHOR = {Passi, Inder Bir Singh and Singh, Mahender and Yadav, Manoj
              Kumar},
     TITLE = {Automorphisms of finite groups},
    SERIES = {Springer Monographs in Mathematics},
 PUBLISHER = {Springer, Singapore},
      YEAR = {2018},
     PAGES = {xix+217},
      ISBN = {978-981-13-2894-7},
   MRCLASS = {20D45 (20D15 20E18 20J05)},
  MRNUMBER = {3887655},
MRREVIEWER = {Andrea\ Caranti},
       DOI = {10.1007/978-981-13-2895-4}
}

@article {zbMATH03549232,
    AUTHOR = {Sah, Chih-han},
     TITLE = {Automorphisms of finite groups},
   JOURNAL = {Journal of Algebra},
    VOLUME = {10},
      YEAR = {1968},
     PAGES = {47--68},
      ISSN = {0021-8693},
   MRCLASS = {20.22},
  MRNUMBER = {229713},
MRREVIEWER = {M.\ F.\ Newman},
       DOI = {10.1016/0021-8693(68)90104-X}
}

\end{document}